\documentclass{lms_arxiv}
\usepackage{amsmath,amsfonts,amssymb}
\usepackage{eucal}
\def\Rp{{\mathbb R^+}}
\def\Z{{\mathbb Z}}
\def\N{{\mathbb N}}
\def\Zp{{\mathbb Z^+}}
\def\R{{\mathbb R}}
\def\P{{\mathbb P}}
\def\E{{\mathbb E}}
\def\I{{\mathbb I}}
\newtheorem{theorem}{Theorem}

\newtheorem{corollary}{Corollary}
\newtheorem{proposition}{Proposition}
\newtheorem{example}{Example}

\title[Harmonic functions for transition kernels]
{Harmonic functions and stationary distributions for
asymptotically homogeneous transition kernels on $\Zp$}

\author{D. Denisov, D. Korshunov, and V. Wachtel}

\classno{31C05, 60J10, 60F10 (primary), 39A10 (secondary)}

\begin{document}
\maketitle

\begin{abstract}
We suggest a method for constructing positive harmonic functions
for a wide class of transition kernels on $\Zp$. We also find
natural conditions under which these functions have positive
finite limits at infinity. Further, we apply our results on harmonic
functions to asymptotically homogeneous Markov chains on $\Zp$
with asymptotically negative drift. More precisely, assuming that
Markov chain satisfy Cram\'er's condition, we study the tail
asymptotics of the stationary distribution. In particular,
we clarify the influence of the rate of convergence of jumps of
the chain towards the limiting distribution.

\end{abstract}

\section{Introduction}
\label{sec:introduction}

\noindent Let $Q$ be a nonnegative finite transition kernel
on $\Zp$, that is, $Q(i,j)\ge 0$ and
$$
Q(i,\Zp)=\sum_{j=0}^\infty Q(i,j)<\infty
\quad\mbox{ for every }i\in\Zp.
$$
We additionally assume that $Q(i,\Zp)>0$ for every $i\in\Zp$.
This kernel is also assumed to be irreducible in a sense that,
for every $i$ and $j$, there exists $n$ such that $Q^n(i,j)>0$.

The function $u(i)$ is called {\it harmonic} if $Qu=u$ which means
$$
\sum_{j=0}^\infty Q(i,j)u(j)=u(i)\quad\mbox{ for every }i\in\Zp.
$$
In this paper we only consider {\it nonnegative} harmonic functions.
Pruitt \cite{Pruitt} has found sufficient and necessary conditions for
existence of such a function. But these conditions are quite hard
to verify. Furthermore, his results do not give any information on the
limiting, as $i\to\infty$, behaviour of harmonic functions. Since this
information is important for the study of asymptotic properties of
Markov chains (see, for example, Foley and McDonald \cite{FD}),
we are interested in a constructive approach to harmonic functions
which would allow to determine their asymptotics.

Every transition kernel can be seen as a combination of a stochastic
transition kernel and of a total mass evolution. Indeed, if we
consider the following stochastic transition kernel:
$$
P(i,j):=\frac{Q(i,j)}{Q(i,\Zp)},
$$
and a Markov chain $X_n$ whose transition probabilities are $P(i,j)$,
then we have, for all $n\ge 1$ and $i$, $j\in\Zp$,
$$
Q^n(i,j)=\E_i\left[\prod_{k=0}^{n-1}Q(X_k,\Zp);X_n=j\right];
$$
hereinafter $\E_i$ means the expectation given $X_0=i$.

We call $X_n$ an {\it underlying} Markov chain.

Assume that, for every state $i$,
\begin{equation}\label{cond.1}
\E_i\prod_{n=0}^\infty \max(Q(X_n,\Zp),1)<\infty.
\end{equation}
This condition makes the following function correctly defined:
\begin{equation}\label{def.f}
f(i):=\E_i\prod_{n=0}^\infty Q(X_n,\Zp)\in[0,\infty).
\end{equation}
Under the condition \eqref{cond.1}, the function $f$ is
a harmonic function for the kernel $Q$. Indeed,
it follows by the conditioning on $X_1$:
\begin{eqnarray*}
f(i) &=& Q(i,\Zp)\sum_{j=0}^\infty P(i,j)
\E\biggl\{\prod_{n=1}^\infty Q(X_n,\Zp)\big|X_1=j\biggr\}\\
&=& \sum_{j=0}^\infty Q(i,j) f(j).
\end{eqnarray*}

In the next section we answer, in particular, the following question.
What are natural conditions that are sufficient for \eqref{cond.1}
in the case when $Q$ is transient?
These sufficient conditions are presented in Proposition
\ref{th:f.exists} and they guarantee that
$\limsup_{i\to\infty}f(i)\le 1$.
Another question is what conditions guarantee that $f(i)$
is positive function and, moreover, $\liminf_{i\to\infty}f(i)\ge 1$.
It is answered in Proposition \ref{th:f>0} in the next section.
Combining these statements we found sufficient conditions for
the existence of a harmonic function satisfying $\lim_{i\to\infty}f(i)=1$.

The expression \eqref{def.f} for the harmonic function $f(i)$
originates from the following two particular cases.
The first simple particular case is provided by stochastic
kernel $Q$ where we have harmonic function $f(i)\equiv 1$
which is the unique (up to a multiple) bounded harmonic function
for recurrent Markov kernels, see Meyn and Tweedie (1993, Theorem 17.1.5).

The second case is a kernel $Q$ which is obtained from some
stochastic kernel $P$ of a Markov chain $Y_n$ by killing
it in some set $B\subset\Zp$, that is,
$$
Q(i,j)=P(i,j)\I\{j\not\in B\},
$$
only defined for those $i$ where $P(i,\Zp\setminus B)>0$.
In this case $Q(i,\Zp)=\P_i\{Y_1\not\in B\}$ and \eqref{def.f}
reads as
\begin{equation}
\label{killed}
f(i)=\P_i\{\tau_B=\infty\},
\end{equation}
where $\tau_B:=\min\{n\ge 1:Y_n\in B\}$. So, if the original Markov
chain $Y_n$ with transition probabilities $P$ is transient and $B$
is finite, then the probability of non-returning to $B$ is
a harmonic function of this Markov chain killed in $B$.
It was proved by Doney in \cite[Theorem 1]{Doney} that
there is the unique harmonic function for a transient random walk
on $\Z$ killed at leaving $\Zp$. This harmonic function was given
as the renewal function generated by descending ladder heights,
which is equal to $\P_i\{\tau_B=\infty\}$ with
$B=\{-1,-2,\ldots\}$, see Section \ref{stay_positive}.

An equivalent way to introduce the condition \eqref{cond.1} is
as follows:
\begin{equation}\label{cond.2}
\E_i e^{\sum_{n=0}^\infty \delta^+(X_n)}<\infty,
\end{equation}
where $\delta(j):=\log Q(j,\Zp)$; hereinafter
$\delta^+:=\max(\delta,0)$ and $\delta^-:=(-\delta)^+$
so that $\delta=\delta^+-\delta^-$.
Then the function $f(i)$ may be also defined as follows:
\begin{equation}\label{def.f.l}
f(i):=\E_ie^{\sum_{j=0}^\infty\ell(j)\delta(j)},
\end{equation}
where $\ell(j)$ is the local time,
$\ell(j):=\sum_{n=0}^\infty \I\{X_n=j\}$.

Our primary motivation for studying harmonic functions of transition
kernels comes from the asymptotic analysis of tail behaviour
of stationary measures of Markov chains. The standard tool
for studying large deviations is an exponential change of measure
(Cram\'er transform). If we follow this approach in the case
of Markov chain, then we get a positive transition kernel
which is not stochastic, in general. In Section \ref{sec:prMc},
we show how the results on asymptotics for harmonic functions obtained
in Section \ref{sec:harmonic} can be used in the study of stationary
measures of asymptotically homogeneous Markov chains.

If the jumps of a positive recurrent Markov chain are bounded, then the equation
for its invariant measure can be considered as a system of
linear difference equations.  The asymptotics of fundamental
solutions to these equations can be found using refinements of the
Poincare-Perron theorem, see Elaydi \cite{Elaydi} for details.
However one can not apply these results directly as we are interested
in positive solutions.

\section{On harmonic functions for transient kernel}
\label{sec:harmonic}

We start with the following solidarity property for the kernel
$Q$ related to positivity of the function $f$.

\begin{proposition}\label{pro:all>0}
If $f(i)>0$ for some $i\in\Zp$, then $f(j)>0$ for every $j\in\Zp$.
\end{proposition}

\begin{proof}
Assume that $f(j)=0$ for some $j\in\Zp$.
The irreducibility of $Q$ implies that there exist a time $N$ and
a path $i_1$, \ldots, $i_{N-1}$, $i_N=i$ such that
$$
\P_j\{X_1=i_1,\ldots,X_{N-1}=i_{N-1},X_N=i\}>0.
$$
Then
\begin{eqnarray*}
f(j) &\ge& Q(j,i_1)\prod_{n=1}^{N-1} Q(i_n,i_{n+1})
\P_j\{X_1=i_1,\ldots,X_{N-1}=i_{N-1},X_N=i\}\\
&&\hspace{50mm}\times
\E\biggl\{\prod_{k=0}^\infty Q(X_{N+k},\Zp)\big|X_N=i\biggr\}\\
&=& f(i)Q(j,i_1)\prod_{n=1}^{N-1} Q(i_n,i_{n+1})
\P_j\{X_1=i_1,\ldots,X_{N-1}=i_{N-1},X_N=i\},
\end{eqnarray*}
which is the product of positive quantities. This contradicts
the equality $f(j)=0$ and the proof is complete.
\end{proof}

\begin{proposition}\label{th:f>0}
Suppose that
\begin{eqnarray}\label{delta-}
\sum_{i=0}^\infty \delta^-(i) &<& \infty,
\end{eqnarray}
that, for every fixed $N$,
\begin{eqnarray}\label{lln}
\P_i\{X_n>N\mbox{ for all }n\ge 0\} &\to& 1
\quad\mbox{as }i\to\infty
\end{eqnarray}
and that
\begin{eqnarray}\label{local.times.E}
\sup_{i\in\Zp}\E_i\ell(i) &<& \infty.
\end{eqnarray}
Then
$$
\liminf_{i\to\infty}f(i)\ge 1.
$$
In particular,
$$
\inf_{i\in\Zp}f(i)>0.
$$
\end{proposition}

\begin{proof}
Since the function $e^x$ is convex, by Jensen's inequality,
\begin{eqnarray*}
f(i) &\ge& \E_i e^{-\sum_{j=0}^\infty\delta^-(j)\ell(j)}
\ge e^{-\sum_{j=0}^\infty \delta^-(j)\E_i\ell(j)}.
\end{eqnarray*}
By the Markov property, for every $j\le N$,
\begin{eqnarray*}
\E_i\ell(j) &\le& \P_i\{X_n\le N\mbox{ for some }n\ge 0\} \E_j\ell(j).
\end{eqnarray*}
Therefore, for every fixed $N$,
by the conditions \eqref{lln} and \eqref{local.times.E},
\begin{eqnarray}\label{lower.le.N}
\sum_{j=0}^N\delta^-(j)\E_i\ell(j) &\to& 0
\quad\mbox{as }i\to\infty.
\end{eqnarray}
On the other hand,
\begin{eqnarray*}
\sum_{j=N+1}^\infty \delta^-(j)\E_i\ell(j) &\le&
\sum_{j=N+1}^\infty \delta^-(j)\E_j\ell(j) \to 0
\quad\mbox{as }N\to\infty,
\end{eqnarray*}
due to the conditions \eqref{delta-} and \eqref{local.times.E}.
Together with \eqref{lower.le.N} this implies the convergence
\begin{eqnarray*}
\sum_{j=0}^\infty \delta^-(j)\E_i\ell(j) &\to& 0
\quad\mbox{as }i\to\infty,
\end{eqnarray*}
which completes the proof.
\end{proof}

\begin{proposition}\label{th:f.exists}
Suppose that the sequence $\delta^+(j)$ is summable, that is,
\begin{eqnarray}\label{delta+}
\delta:=\sum_{j=0}^\infty \delta^+(j) &<& \infty.
\end{eqnarray}
Then, for every $i\in\Zp$,
\begin{eqnarray*}
f(i) &\le& \sup_{j\in\Zp}\E_i e^{\delta\ell(j)}\le\infty.
\end{eqnarray*}
If, in addition, the condition \eqref{lln} holds and
\begin{eqnarray}\label{sup.delta}
\sup_{i\in\Zp}\E_i e^{\delta\ell(i)} &<& \infty,
\end{eqnarray}
then
$$
\limsup_{i\to\infty}f(i)\le 1.
$$
\end{proposition}

\begin{proof}
Take $r(j):=\frac{\delta}{\delta^+(j)}$ so that
$\sum_{j=0}^\infty \frac{1}{r(j)}=1$. We have
\begin{eqnarray*}
f(i) &\le& \E_i \prod_{j=0}^\infty e^{\delta^+(j)\ell(j)}
= \E_i\prod_{j=0}^\infty e^{\frac{\delta}{r(j)}\ell(j)}.
\end{eqnarray*}
Apply H\"older's inequality:
\begin{eqnarray*}
\E_i\prod_{j=0}^\infty e^{\frac{\delta}{r(j)}\ell(j)}
&\le& \prod_{j=0}^\infty \Bigl(\E_i e^{\delta\ell(j)}\Bigr)^{1/r(j)}.
\end{eqnarray*}
Therefore,
\begin{eqnarray}\label{Holder}
f(i) &\le& \prod_{j=0}^\infty(\E_i e^{\delta\ell(j)})^{\delta^+(j)/\delta}\\
&\le& \Bigl(\sup_{j\in\Zp}\E_i e^{\delta\ell(j)}\Bigr)
^{\sum_{j=0}^\infty\delta^+(j)/\delta},\nonumber
\end{eqnarray}
which yields the upper bound for $f(i)$ in terms of exponential
moments of local times, by \eqref{delta+}.

Now turn to more precise upper bound. Fix some $N$ and
rewrite \eqref{Holder} as
\begin{eqnarray*}
f(i) &\le&
\prod_{j=0}^N(\E_i e^{\delta\ell(j)})^{\delta^+(j)/\delta}
\times \prod_{j=N+1}^\infty(\E_i e^{\delta\ell(j)})^{\delta^+(j)/\delta}.
\end{eqnarray*}
By the Markov property, for every fixed $j\le N$,
\begin{eqnarray*}
\E_i e^{\delta\ell(j)} &\le& \P_i\{X_n>N\mbox{ for all }n\ge 0\}\\
&& +\P_i\{X_n\le N\mbox{ for some }n\ge 0\}
\E_j e^{\delta\ell(j)}
\to 1\quad\mbox{ as }i\to\infty,
\end{eqnarray*}
by the conditions \eqref{lln} and \eqref{sup.delta}.
Therefore, for every fixed $N$,
\begin{eqnarray}\label{Ei1}
\prod_{j=0}^N(\E_i e^{\delta\ell(j)})^{\delta^+(j)/\delta}
&\to& 1\quad\mbox{ as }i\to\infty.
\end{eqnarray}
The second product possesses the following upper bound:
\begin{eqnarray*}
\prod_{j=N+1}^\infty(\E_i e^{\delta\ell(j)})^{\delta^+(j)/\delta}
&\le& \prod_{j=N+1}^\infty(\E_j e^{\delta\ell(j)})^{\delta^+(j)/\delta}\\
&\le& \Bigl(\sup_{j\in\Zp}\E_j e^{\delta\ell(j)}\Bigr)
^{\sum_{j=N+1}^\infty \delta^+(j)/\delta}.
\end{eqnarray*}
Since $\sum_{j=N+1}^\infty\delta^+(j)/\delta\to 0$ as $N\to\infty$,
by \eqref{sup.delta},
\begin{eqnarray*}
\prod_{j=N+1}^\infty(\E_i e^{\delta\ell(j)})^{\delta^+(j)/\delta}
&\to& 1
\end{eqnarray*}
as $N\to\infty$ uniformly in $i\in\Zp$, which together with
\eqref{Ei1} yields the asymptotic upper bound for $f(i)$.
\end{proof}

Propositions \ref{th:f>0} and \ref{th:f.exists} yield the following result.

\begin{theorem}\label{th:f.to.1}
Suppose that
\begin{eqnarray}\label{delta}
\sum_{i=0}^\infty |\delta(i)| &<& \infty,
\end{eqnarray}
that the condition \eqref{lln} holds and that
\begin{eqnarray}\label{local.times.exp}
\sup_{i\in\Zp}\E_i e^{\delta\ell(i)} &<& \infty,
\end{eqnarray}
where
\begin{eqnarray*}
\delta &:=& \sum_{i=0}^\infty \delta^+(i).
\end{eqnarray*}
Then the function $f$ is harmonic and $f(i)\to 1$ as $i\to\infty$.
\end{theorem}

Our construction of a harmonic function is alternative to the
construction of Foley and McDonald, see \cite[Proposition 2.1]{FD}.
Their analysis is based on the assumption that the series
$$
\sum_{n=1}^\infty Q^n(i,j)z^n
$$
has the (common for all $i$ and $j$) radius of convergence $R$
bigger than $1$. It seems to be quite difficult to compare
this assumption with our condition \eqref{local.times.exp}.
Clearly, the condition \eqref{local.times.exp} is ready
for verification in particular cases because the total
masses and the embedded Markov chain are factorised in it.
Also, our condition \eqref{delta} is weaker than the
`closeness' condition in \cite{FD}.

\begin{example}\label{ex:1}
Let $Q$ be the following local perturbation at the origin of
the transition kernel of a simple random walk on $\Zp$:
\begin{align*}
&Q(0,1)=\alpha>0,\\
&Q(i,i+1)=p>1/2,\quad Q(i,i-1)=1-p=:q,\quad i\ge 1.
\end{align*}
\end{example}
Then we have $Q(0,\Zp)=\alpha$ and $Q(i,\Zp)=1$ for all $i\ge 1$.
In other words, $\delta(0)=\log\alpha$ and $\delta(i)=0$ for $i\ge 1$,
so that $\delta=\log\alpha$.
The underlying Markov chain $X_n$ is a simple random walk with
reflection at zero. More precisely, its transition kernel is given by
\begin{align*}
&P(0,1)=1\\
&P(i,i+1)=p,\ P(i,i-1)=q,\quad i\ge 1.
\end{align*}
According to Theorem \ref{th:f.to.1}, the condition
$\E_0\alpha^{\ell(0)}<\infty$ implies that the function
$f(i)=\E_i \alpha^{\ell(0)}$ is a positive harmonic function
with $f(i)\to 1$ as $i\to\infty$. The local time $\ell(0)$
is geometrically distributed with the parameter $q/p$, that is,
$$
\P_0\{\ell(0)=k\}=(1-q/p)(q/p)^{k-1},\quad k\ge 1.
$$
Therefore,
\begin{equation}\label{ex1.1}
f(0)=\E_0\alpha^{\ell(0)}=\alpha\frac{1-q/p}{1-\alpha q/p}<\infty\quad
\text{if and only if }\alpha<p/q.
\end{equation}
Moreover, for every $i\ge 1$,
\begin{align}\label{ex1.2}
\nonumber
f(i)=\E_i \alpha^{\ell(0)}
&=\P_i\{X_n\neq0\text{ for all }n\ge 1\}+\P_i\{X_n=0
\text{ for some }n\ge 1\}\E_0\alpha^{\ell(0)}\\
&=1-(q/p)^i+(q/p)^if(0).
\end{align}

Since a harmonic function $f(i)$ for $Q$ is a solution to the
system of equations
\begin{align*}
&f(0)=\alpha f(1),\\
&f(i)=pf(i+1)+qf(i-1),\quad i\ge 1,
\end{align*}
we may determine it using standard methods from the theory of difference
equations. Indeed, equations for $i\ge 1$ can be rewritten as follows
$$
f(i+1)-f(i)=\frac{q}{p}(f(i)-f(i-1)).
$$
Consequently,
\begin{align*}
f(i)-f(0)&=\sum_{j=0}^{i-1}(f(j+1)-f(j))
=(f(1)-f(0))\sum_{j=0}^{i-1}(q/p)^j\\
&=(f(1)-f(0))\frac{1-(q/p)^i}{1-q/p}.
\end{align*}
Noting that $f(1)=f(0)/\alpha$ we get
\begin{equation}\label{ex1.3}
f(i)=f(0)\biggl[1+(1/\alpha-1)\frac{1-(q/p)^i}{1-q/p}\biggr],
\quad i\ge 1.
\end{equation}
Choosing $f(0)$ as in \eqref{ex1.1}, we conclude that the expressions
in \eqref{ex1.2} and \eqref{ex1.3} are equal for all $\alpha<p/q$.
Further, for every $\alpha>p/q$ and every $f(0)>0$, the function
$f(i)$ from \eqref{ex1.3} becomes negative for $i$ large.
Therefore, there is no a positive harmonic function for $\alpha>p/q$.
Finally, in the critical case $\alpha=p/q$, we have $f(i)=f(0)(q/p)^i.
$\hfill$\Box$

\begin{example}
Consider the transition kernel  given by the following relations:
\begin{align*}
&Q(i,i+1)=\alpha_i>1,\quad i=0,\ldots,N-1,\\
&Q(N,N+1)=p,\quad Q(N,0)=q,\\
&Q(i,i+1)=p,\quad Q(i,i-1)=q,\quad i>N.
\end{align*}
\end{example}

Aggregating the states $0$, \ldots, $N-1$ into a new state,
we obtain the transition kernel from Example \ref{ex:1} with
$\alpha=\alpha_0\ldots\alpha_{N-1}$. Therefore, there exists
a positive harmonic function $f$ with $f(i)\to 1$ as $i\to\infty$
if an only if $p/q>\alpha_0\ldots\alpha_{N-1}$.
But this is equivalent to
$$
\sup_{i<N}e^{\delta\ell(i)}<\infty\quad
\mbox{where }\delta=\sum_{i=0}^{N-1}\log \alpha_i.
$$
This shows that exponential moment assumption on the local times
in Theorem \ref{th:f.to.1} is quite close to the necessary one.
\hfill$\Box$

Next we give simple sufficient conditions that guaranties
finiteness of some exponential moments for local times.

\begin{proposition}\label{pro:suffi.for.ell.m}
Suppose that there exists a random variable $\eta$ such that
$\E\eta>0$ and, for all $i\in\Zp$ and $j\in\Z$,
\begin{eqnarray}\label{minorant.eta}
\P\{\eta>j\} &\le& \P_i\{X_1-X_0>j\},
\end{eqnarray}
that is, $\eta$ is a stochastic minorant for jumps of the chain
$X_n$ at every state. Then
\begin{eqnarray}\label{local.times.exp.exists}
\sup_{i\in\Zp}\E_i e^{\gamma\ell(i)} &<& \infty
\quad\mbox{for some }\gamma>0.
\end{eqnarray}
\end{proposition}

\begin{proof}
The relation \eqref{local.times.exp.exists}
will follow if we prove that
\begin{eqnarray}\label{lower.est.gen}
p:=\inf_{i\in\Zp}\P_i\{X_n\ge i+1\mbox{ for all }n\ge 0\} &>& 0,
\end{eqnarray}
because then every local time $\ell(i)$ satisfies
$$
\P\{\ell(i)\ge k+1\}\le (1-p)^k\quad\mbox{ where }1-p<1.
$$
Indeed, by the condition \eqref{minorant.eta}, for every $i$,
we may construct $X_0=i$, $X_1$, \ldots\ and independent copies
$\eta_1$, $\eta_2$, \ldots\ of $\eta$ on some probability space
in such a way that
\begin{eqnarray*}
X_n &\ge& i+\eta_1+\ldots+\eta_n\quad\mbox{for all }n\ge 1.
\end{eqnarray*}
Since $\E\eta>0$, the strong law of large numbers implies that
\begin{eqnarray*}
\P\{\eta_1+\ldots+\eta_n\ge 1\mbox{ for all }n\ge 1\} &>& 0.
\end{eqnarray*}
Altogether yields \eqref{lower.est.gen} with $\gamma<\log\frac{1}{1-p}$.
\end{proof}

The latter result may be generalised for the case where there is no
a minorant general for all jumps but there is everywhere positive drift.
In order to produce this generalisation we first need
the following statement.

\begin{proposition}\label{pro:suffi.for.ell}
Assume that, for every $i\in\Zp$, there exists a positive monotone
decreasing function $g_i(j)$ such that $g_i(X_n)$ is a supermartingale
and such that
$$
p_1:=\sup_{i\ge 1}\frac{g_i(i)}{g_i(i-1)}<1.
$$
If, in addition,
$$
p_2:=\inf_{i\ge 0}\P_i\{X_1\ge i+1\}>0,
$$
then \eqref{local.times.exp.exists} holds.
\end{proposition}

\begin{proof}
For every $i\ge 1$, applying Doob's inequality to the supermartingale
$g_i(X_n)$ with $X_0=i$, we obtain that
\begin{align*}
\P_i\Bigl\{\inf_{n\ge 1}X_n\le i-1\Bigr\}
&=\P_i\Bigl\{\sup_{n\ge 1}g_i(X_n)\ge g_i(i-1)\Bigr\}\\
&\le \frac{\E_i g_i(X_0)}{g_i(i-1)}=\frac{g_i(i)}{g_i(i-1)}\le p_1.
\end{align*}
Therefore,
\begin{align*}
\P_i\Bigl\{\inf_{n\ge 1}X_n\ge i+1\Bigr\}
&\ge\sum_{j=i+1}^\infty\P_i\{X_1=j\}\Bigl(1-\P_j\Bigl\{\inf_{n\ge 1}X_n\le i\Bigr\}\Bigr)\\
&\ge(1-p_1)\sum_{j>i}\P_i\{X_1=j\}\\
&\ge (1-p_1)p_2>0
\end{align*}
uniformly in $i\in\Zp$. Then
$$
\P\{\ell(i)\ge k+1\}\le (1-p)^k\quad\mbox{ where }p:=(1-p_1)p_2>0
$$
and we obtain \eqref{local.times.exp.exists} with any
$\gamma<\log\frac{1}{1-p}$.
\end{proof}

The latter proposition allows to deduce finiteness of exponential moments
of local times for Markov chains with everywhere positive drift.

\begin{proposition}\label{DF}
Assume that there exist $\varepsilon>0$ and $M<\infty$ such that
$$
\E_i\{X_1-i;X_1-i\le M\}\ge \varepsilon\ \text{ for all }i\in\Zp.
$$
Assume also that there exists $\zeta\ge 0$ with $\E\zeta<\infty$ such that
$$
\P_i\{X_1-i\le -j\}\le \P\{\zeta\ge j\}\ \text{ for all }i,j\in\Zp.
$$
Then \eqref{local.times.exp.exists} holds.
\end{proposition}

\begin{proof}
Since $\E\zeta<\infty$, there exists a non-increasing integrable
function $f_1(x)$, $x\in\Rp$, such that $\P\{\zeta>x\}=o(f_1(x))$ as $x\to\infty$.
In its turn, by \cite{D2006}, there exists a continuous non-increasing
integrable regularly varying at infinity with index $-1$ function $f_2(x)$
such that $f_1(x)\le f_2(x)$. Since the non-increasing function $f_2$ is
regularly varying at infinity with index $-1$, there exists
a sufficiently small $\delta>0$ and a sufficiently large $T$ such that
\begin{eqnarray}\label{f.cond.1}
f_2(1+(1-\delta)t)-f_2(1+t+M) &\le& \frac{\varepsilon}{2\E\zeta}f_2(1+t+M)
\quad \mbox{ for all }t\ge T
\end{eqnarray}
and
\begin{eqnarray}\label{f.cond.2}
\P\{\zeta\ge\delta t\}\int_0^\infty f_2(y)dy
&\le& \frac{\varepsilon}{2}f_2(1+t+M)
\quad \mbox{ for all }t\ge T.
\end{eqnarray}
Now take
$$
f(t):=\min(f_2(T+M),f_2(t))=\left\{
\begin{array}{ll}
f_2(T+M) & \mbox{ if }t\in[0,T+M],\\
f_2(t) & \mbox{ if }t\ge T+M
\end{array}
\right.
$$
and
$$
g(x):=\int_{(x+1)^+}^\infty f(y)dy,\quad x\in\R.
$$
The function $g(x)$ is positive and decreasing, $g(-1)<\infty$.
By \eqref{f.cond.1} and \eqref{f.cond.2}, the function $f$ satisfies
\begin{eqnarray}\label{f.cond.3}
f(1+(1-\delta)t) &\le& (1+\varepsilon/2\E\zeta)f(1+t+M)
\quad \mbox{ for all }t\ge 0,\\
\label{f.cond.4}
g(-1)\P\{\zeta\ge\delta t\} &\le& \frac{\varepsilon}{2}f(1+t+M)
\quad \mbox{ for all }t\ge T.
\end{eqnarray}
Define $g_i(j):=g(j-i)$. By the construction,
$$
\frac{g_i(i)}{g_i(i-1)}=\frac{g(0)}{g(-1)}
=\frac{\int_1^\infty f(y)dy}{\int_0^\infty f(y)dy}<1.
$$
Then it remains to prove that,
for every $i\in\Zp$, $g_i(X_n)$ is a supermartingale, that is,
$\E_{i+j} g_i(X_1)\le g_i(i+j)$ for every $j\ge -i$. For $j\le -1$,
since $g_i$ is bounded by $g(-1)$,
\begin{eqnarray}\label{gi0}
\E_{i+j} g_i(X_1)-g_i(i+j) &=&
\E_{i+j} g_i(X_1)-g(-1) \le 0.
\end{eqnarray}

Next consider the case $0\le j\le T-1$. In this case,
$f(0)=f(1+j+M)$ which implies that $g_i(x)$ is linear when
$x\in[-i-1,i+j+M]$. Then
\begin{eqnarray}\label{gi1}
\E_{i+j} g_i(X_1)-g_i(i+j) &\le&
\E_{i+j}\{g_i(X_1)-g_i(i+j);X_1-i-j\le M\}\nonumber\\
&=& g_i'(i)\E_{i+j}\{X_1-i-j;X_1-i-j\le M\}\nonumber\\
&=& -f(0)\E_{i+j}\{X_1-i-j;X_1-i-j\le M\}\nonumber\\
&\le& -f(0)\varepsilon<0.
\end{eqnarray}

Now consider the case $j\ge T$.
Since $g_i$ is bounded by $g(-1)$ and non-increasing,
\begin{eqnarray}\label{gi2}
\E_{i+j} g_i(X_1)-g_i(i+j) &\le& g(-1)\P_{i+j}\{X_1\le i+j-\delta j\}\nonumber\\
&&\hspace{10mm}+\E_{i+j}\{g_i(X_1)-g_i(i+j);-\delta j<X_1-i-j\le 0\}\nonumber\\
&&\hspace{20mm} +\E_{i+j}\{g_i(X_1)-g_i(i+j);0<X_1-i-j\le M\}\nonumber\\
&=& E_1+E_2+E_3.
\end{eqnarray}
We have, by \eqref{f.cond.4},
\begin{eqnarray}\label{gi21}
E_1 &\le& g(-1)\P\{\zeta\ge \delta j\} \le \frac{\varepsilon}{2}f(1+j+M).
\end{eqnarray}
Since $g_i'(x)=-f(x+1-i)$ for $x\ge i-1$,
the second term possesses the following estimate:
\begin{eqnarray*}
E_2 &\le& g_i'(i+j-\delta j)\E_{i+j}\{X_1-i-j;-\delta j<X_1-i-j\le 0\}\\
&=& -f(1+j-\delta j)\E_{i+j}\{X_1-i-j;-\delta j<X_1-i-j\le 0\}\\
&\le& -f(1+j-\delta j)\E_{i+j}\{X_1-i-j;X_1-i-j\le 0\}.
\end{eqnarray*}
Therefore, due to \eqref{f.cond.3},
\begin{eqnarray}\label{gi22}
E_2 &\le& -(1+\varepsilon/2\E\zeta)f(1+j+M)\E_{i+j}\{X_1-i-j;X_1-i-j\le 0\}\nonumber\\
&\le& -f(1+j+M)\E_{i+j}\{X_1-i-j;X_1-i-j\le 0\}
+\frac{\varepsilon}{2}f(1+j+M).
\end{eqnarray}
The third term is not greater than
\begin{eqnarray}\label{gi23}
E_3 &\le& g_i'(i+j+M)\E_{i+j}\{X_1-i-j;0<X_1-i-j\le M\}\nonumber\\
&=& -f(1+j+M)\E_{i+j}\{X_1-i-j;0<X_1-i-j\le M\}.
\end{eqnarray}
Substituting \eqref{gi21}--\eqref{gi23} into \eqref{gi2} we get
the desired inequality $\E_{i+j} g_i(X_1)-g_i(i+j) \le 0$, for $j\ge T$.
Together with \eqref{gi0} and \eqref{gi1} this proves
that $g_i(X_n)$ constitutes a nonnegative bounded supermartingale.
The proof of the proposition is complete.
\end{proof}

\section{Random walk with negative drift conditioned to stay nonnegative}
\label{stay_positive}

Consider the simplest application of our method of construction of harmonic functions.
It deals with random walk conditioned to stay nonnegative.
Let $S_0=0$, $S_n=\sum_{k=1}^n\xi_k$ be a random walk with independent
identically distributed jumps, $\E\xi_k<0$. One of the possible ways to
define a random walk conditioned to stay nonnegative consists in performing
Doob's $h$-transform over $S_n$ killed at leaving $\Zp$, that is, the Markov
chain on $\Zp$ with the transition probabilities
$$
P(i,j)=\frac{f(j)}{f(i)}\P\{i+\xi_1=j\},\  i,j\in\Zp,
$$
where $f$ is a positive harmonic function $f(i)$ for the killed random walk, that is,
$$
f(i)=\sum_{j=0}^\infty \P\{\xi_1=j-i\}f(j),\quad i\ge 0.
$$
According to Theorem 1 of \cite{Doney} such a function exists
if and only if
$$
\E e^{\beta\xi_1}=1\ \text{ for some }\beta>0.
$$
This function is unique (up to a constant multiplier) and is defined
in \cite{Doney} as
\begin{equation}
\label{doney_repr}
f(i)=\sum_{j=0}^i e^{\beta(i-j)}u(j),
\end{equation}
where $u(j)$ stands for the mass function of the renewal process
of strict descending ladder heights in $S_n$.

Now let us show how our approach provides
another representation of the harmonic function $f(i)$.
Start with the following transition kernel on $\Zp$:
$$
Q(i,j):=e^{(j-i)\beta}\P\{\xi_1=j-i\},\quad i,j\in\Zp.
$$
This kernel represents transition probabilities for the random
walk $S_n^{(\beta)}$ killed at leaving $\Zp$, where $S_n^{(\beta)}$
is the result of exponential change of measure with parameter $\beta$,
that is,
\begin{align*}
&S_0^{(\beta)}=0,\quad S_n^{(\beta)}=\sum_{k=1}^n \xi^{(\beta)}_k\\
&\P\{\xi^{(\beta)}_k=j\}=e^{\beta j}\P\{\xi_k=j\},\quad j\in\Z.
\end{align*}
As we have
already mentioned in the introduction, see \eqref{killed},
$$
f^*(i):=\P\{i+S^{(\beta)}_n\ge 0\text{ for all }n\ge 1\}
=\P\Bigl\{\min_{n\ge 0}S^{(\beta)}_n\ge -i\Bigr\}
$$
is harmonic for the kernel $Q$. Hence, the function
\begin{equation}
\label{harm.function}
f(i):=e^{\beta i}f^*(i)
=e^{\beta i}\P\Bigl\{\min_{n\ge 0}S^{(\beta)}_n\ge -i\Bigr\}
\end{equation}
is harmonic for the random walk $S_n$
killed at leaving $\Zp$. Notice that this harmonic function
possesses the following lower and upper bounds:
\begin{equation}
\label{harm.bounds}
e^{\beta i}-e^{-\beta}\le f(i)\le e^{\beta i}.
\end{equation}
The upper bound immediately follows from $f^*(i)\le 1$.
The lower bound follows by the Cram\'er--Lundberg estimate,
$$
\P\Bigl\{\min_{n\ge 0}S^{(\beta)}_n\le -i-1\Bigr\}\le e^{-\beta (i+1)}.
$$
Additionally to \eqref{harm.bounds} notice that, by the Cram\'er--Lundberg
asymptotics,
$$
f(i)-e^{\beta i}\to C\in(-e^{-\beta},0)\text{ as }i\to\infty.
$$

Let us show that the functions defined in \eqref{doney_repr} and \eqref{harm.function}
coinside up to a multiplicative constant. Indeed,
let $(\tau_k,\chi_k)$ and $(\tau^{(\beta)}_k,\chi^{(\beta)}_k)$ denote
descending ladder processes for $S_n$ and $S_n^{(\beta)}$, respectively.
It follows from the definition of $S_n^{(\beta)}$ that
$$
\P\{\chi^{(\beta)}_k=x,\tau^{(\beta)}_k=j\}
=e^{-\beta x}\P\{\chi_k=x,\tau_k=j\}
$$
for all $x,j>0$. Therefore,
\begin{align*}
\P\Bigl\{\min_{n\ge 0}S^{(\beta)}_n=-l\Bigr\}
&=\sum_{k=0}^\infty\P\Biggl\{\sum_{j=0}^k\chi_j^{(\beta)}=l,\tau_1^{(\beta)}<\infty,\ldots,
\tau_k^{(\beta)}<\infty,\tau_{k+1}^{(\beta)}=\infty\Biggr\}\\
&=\P\{\tau_1^{(\beta)}=\infty\}\sum_{k=0}^\infty\P\Biggl\{\sum_{j=0}^k\chi_j=l\Biggr\}e^{-\beta l}\\
&=\P\{\tau_1^{(\beta)}=\infty\}e^{-\beta l}u(l).
\end{align*}
This gives the desired equivalence with the multiplier
$\P\{\tau_1^{(\beta)}=\infty\}=1-\E e^{-\beta\chi_1}$.

The random walk conditioned to stay nonnegative is the simplest
Markov chain where the general scheme of construction of
a harmonic function helps. In the next section we follow almost the same
techniques in our study of tail behavior for the asymptotically
homogeneous in space Markov chains with negative drift under Cram\'er's
type assumptions. Although the scheme will be the same in main aspects,
the associated additional calculations turn out to be more complicated.

\section{Positive recurrent Markov chains: asymptotic behaviour of
stationary distributions}
\label{sec:prMc}

In this section we consider a positive recurrent Markov chain $X_n$
on $\Zp$ with stationary probabilities $\pi(i)>0$, that is,
$$
\sum_{i=0}^\infty \pi(i)=1\quad\mbox{ and }\quad
\pi(i)=\sum_{j=1}^\infty \pi(j)P(j,i)\quad\mbox{for every }i\in\Zp,
$$
where $P(j,i)$ is transition probability from $j$ to $i$.
We are interested in the asymptotics of $\pi(i)$ as $i\to\infty$
in the case where the distribution of $X_n$ has some positive
exponential moments finite, more precisely, in the so-called
Cram\'er case. Let $\xi(i)$ denote a random variable distributed as
the jump of the chain $X_n$ from the state
$i$, that is,
$$
\P\{\xi(i)=j\}=P(i,i+j),\quad j\ge -i.
$$
We shall always assume that $X_n$ is {\it asymptotically homogeneous
in space}, that is,
\begin{equation}\label{asymp.hom}
\xi(i) \Rightarrow \xi\quad\mbox{as }i\to\infty.
\end{equation}
We assume $\E\xi<0$ and that $\Z$ is the lattice with minimal span
for the distribution of $\xi$. By the Cram\'er case we mean the case
where
$$
\E e^{\beta\xi}=1\quad\mbox{and}\quad\E\xi e^{\beta\xi}<\infty
\quad\mbox{for some }\beta>0.
$$

The simplest and one of the most important examples of asymptotically
homogeneous Markov chains is a random walk with delay at zero:
$$
W_{n+1}=(W_n+\zeta_{n+1})^+,\ n\ge 0,
$$
where $\{\zeta_k\}$ are independent copies of $\xi$.
As is well-known, the stationary measure of $W_n$, say $\pi_{W}$, coincides with
the distribution of $\sup_{n\ge 0}\sum_{k=1}^n\zeta_k$. Then, by the
classical results due Cram\'er and Lundberg,
$$
\pi_W(i)\sim c e^{-\beta i}\text{ as }i\to\infty.
$$
Since the jumps of chains $X_n$ and $W_n$ are asymptotically equivalent,
one can expect that the corresponding stationary distributions have
similar asymptotics. This is true on the logarithmic scale only:
Borovkov and Korshunov have shown, see Theorem 3 in \cite{BK1}, that
if $\sup_{i\ge 0}\E e^{\beta\xi(i)}<\infty$, then
$$
\frac{1}{i}\log \pi(i)\to -\beta\text{ as }i\to\infty.
$$

It turns out that the exact (without logarithmic scaling) asymptotic
behaviour of $\pi$ depends not only on the distribution of $\xi$, but
also on the speed of convergence in \eqref{asymp.hom}.

Our next result describes the case when the convergence is so fast
that the measure $\pi$ is asymptotically proportional to the stationary
measure of $W_n$.

\begin{theorem}\label{th:conv}
Suppose that
\begin{eqnarray}\label{Xi}
\xi(i) &\le_{st}& \Xi, \qquad i\in\Zp,
\end{eqnarray}
for some random variable $\Xi$ such that $\E\Xi e^{\beta\Xi}<\infty$
and
\begin{eqnarray}\label{conver}
\sum_{i=0}^\infty |\E e^{\beta\xi(i)}-1| &<& \infty.
\end{eqnarray}
Then $\pi(i)\sim c e^{-\beta i}$ as $i\to\infty$ where $c>0$.
\end{theorem}
It is worth mentioning that \eqref{conver} is weaker than conditions
we found in the literature. First, Borovkov and Korshunov \cite{BK1}
proved exponential asymptotics for $\pi$ under the condition
$$
\sum_{i=0}^\infty\int_{-\infty}^\infty
e^{\beta y}\left|\P\{\xi(i)<y\}-\P\{\xi<y\}\right|dy<\infty,
$$
which is definitely stronger than \eqref{conver} and implies,
in particular, that also the expectations of $\xi^{(\beta)}(i)$ converge with
a summable speed. Furthermore, to show that the constant $c$ in front
of $e^{-\beta i}$ is positive they introduced the following condition:
$$
\sum_{i=0}^\infty\bigl(\E e^{\beta\xi(i)}-1\bigr)^-i\log i<\infty.
$$
Second, Foley and McDonald \cite{FD} used the assumption,
which can be rewritten in our notations as follows
$$
\sum_{i=0}^\infty\sum_{j\in\Z}e^{\beta j}|\P\{\xi(i)=j\}-\P\{\xi=j\}|<\infty.
$$

Furthermore, the condition \eqref{conver} is quite close to the optimal one.
If, for example, $\E e^{\beta\xi(i)}-1$ are of the same sign and not summable,
then $\pi(i)e^{\beta i}$ converges either to zero or to infinity, see
Corollary \ref{cor_power} below. Thus, if \eqref{conver} is violated, then
$\pi(i)$ may have exponential asymptotics only in the case when
$\E e^{\beta\xi(i)}-1$ is changing its sign infinitely often.

\begin{example}\label{ex:3}
Consider a Markov chain $X_n$ which jumps to the next neighbours only:
$$
\P\{\xi(i)=1\}=1-\P\{\xi(i)=-1\}=p+\varphi(i).
$$
\end{example}
Assume that, as $i\to\infty$,
$$
\varphi(i)\sim\left\{
\begin{array}{ll}
i^{-\gamma}, &i=2k\\
-i^{-\gamma}, &i=2k+1
\end{array}
\right.
$$
with some $\gamma\in(1/2,1)$. Clearly, \eqref{conver} is not satisfied.
Let us look at the values of $X_n$ at even time moments, i.e.,
$$
Y_k=X_{2k},\quad k\ge 0.
$$
Then we have
\begin{align*}
&\P_i\{Y_1-i=-2\}=(q-\varphi(i))(q-\varphi(i-1)),\\
&\P_i\{Y_1-i=0\}=(q-\varphi(i))(p+\varphi(i-1))+(p+\varphi(i))(q-\varphi(i+1)),\\
&\P_i\{Y_1-i=2\}=(p+\varphi(i))(p+\varphi(i+1)),
\end{align*}
where $q:=1-p$. From these equalities we obtain
\begin{align*}
&\E_i\left[\left(\frac{q}{p}\right)^{Y_1-i}\right]-1=
\left(\frac{p^2}{q^2}-1\right)\P_i\{Y_1-i=-2\}+\left(\frac{q^2}{p^2}-1\right)\P_i\{Y_1-i=2\}\\
&\hspace{1cm}=\left(\frac{p^2}{q^2}-1\right)(q-\varphi(i))(q-\varphi(i-1))
+\left(\frac{q^2}{p^2}-1\right)(p+\varphi(i))(p+\varphi(i+1))\\
&\hspace{1cm}=-q\left(\frac{p^2}{q^2}-1\right)(\varphi(i)+\varphi(i-1))
+p\left(\frac{q^2}{p^2}-1\right)(\varphi(i)+\varphi(i+1))
+O(i^{-2\gamma}).
\end{align*}
Noting that $\varphi(i)+\varphi(i+1)=O(i^{-\gamma-1})$, we conclude that the sequence
$|\E_i(q/p)^{Y_1-i}-1|$ is summable and, consequently, we may apply Theorem \ref{th:conv}.
Since $\pi$ is stationary also for $Y_n$, we obtain $\pi(i)\sim c(p/q)^i$ as $i\to\infty$.
\hfill$\Box$

\begin{proof}[Proof of Theorem \ref{th:conv}]
Fix some $N\in\Zp$. As well-known
(see, e.g. \cite[Theorem 10.4.9]{MT}) the invariant measure
$\pi$ possesses the equality
\begin{equation}\label{pi-def}
\pi(i)=\sum_{j=0}^N \pi(j)\sum_{n=0}^\infty \P_j\{X_n=i;\tau_N>n\},
\end{equation}
where
$$
\tau_N:=\inf\{n\ge 1:X_n\le N\}.
$$
Let $h(i)$ be a harmonic function for $X_n$ killed in $[0,N]$,
that is, for every $i>N$,
$$
\E_i\{h(X_1);X_1>N\}=h(i).
$$
Then we can perform Doob's $h$-transform on $X_n$ killed in $[0,N]$
and define a new Markov chain $\widehat X_n$ on $\Zp$
with the following transition kernel
$$
\P_i\{\widehat X_1=j\}=\frac{h(j)}{h(i)}\P_i\{X_1=j;\tau_N>1\}
$$
if $h(i)>0$ and $\P_i\{\widehat X_1=j\}$ being arbitrary defined
if $h(i)=0$. Since $h$ is harmonic, then we also have
\begin{align}\label{connection}
\P_i\{\widehat X_n=j\}=\frac{h(j)}{h(i)}\P_i\{X_n=j;\tau_N>n\}
\ \mbox{ for all }n.
\end{align}
Combining (\ref{connection}) and (\ref{pi-def}), we get
\begin{eqnarray}\label{pi.i.pre}
\pi(i) &=& \frac{1}{h(i)}\sum_{j=0}^N\pi(j)h(j)
\sum_{n=0}^\infty \P_j\{\widehat X_n=i\}\nonumber\\
&=& \frac{\widehat U(i)}{h(i)}\sum_{j=0}^N\pi(j)h(j),
\end{eqnarray}
where $\widehat U$ is the renewal measure generated by
the chain $\widehat X_n$ with initial distribution
$$
\P\{\widehat X_0=j\}=\widehat c\pi(j)h(j),\ j\le N,
\quad\text{where }\widehat c:=\frac{1}{\sum_{j=0}^N\pi(j)h(j)}.
$$

Suppose that the harmonic function $h(i)$ is such that the jumps
$\widehat\xi(i)$ of the chain $\widehat X_n$ satisfy the following
conditions:
\begin{eqnarray}\label{hat.as.hom}
\widehat\xi(i)\Rightarrow\widehat\xi\quad\mbox{as }
x\to\infty,\quad \E\widehat\xi>0,
\end{eqnarray}
the family of random variables $\{|\widehat\xi(i)|,\ i\in\Zp\}$
admit an integrable majorant $\widehat\Xi$, that is,
\begin{eqnarray}\label{majoriz}
|\widehat\xi(i)| &\le_{\rm st}& \widehat\Xi
\quad\mbox{for all }i\in\Zp,\quad \E\widehat\Xi<\infty;
\end{eqnarray}
and
\begin{eqnarray}\label{finite.bound.U}
\sup_{i\in\Zp} \widehat U(i) &<& \infty.
\end{eqnarray}
Then the key renewal theorem for asymptotically homogeneous
in space Markov chains from Korshunov \cite{Kor08} states that
$\widehat U(i)\to 1/\E\widehat\xi$ as $i\to\infty$. Substituting
this into \eqref{pi.i.pre} we deduce the following asymptotics
\begin{eqnarray}\label{pi.asy.V}
\pi(i) &\sim& \frac{\sum_{j=0}^N\pi(j)h(j)}{\E\widehat\xi}
\frac{1}{h(i)} \quad\mbox{as }i\to\infty.
\end{eqnarray}

So, now we need to choose a level $N$ and to construct a harmonic
function $h(i)$ for $X_n$ killed in $[0,N]$ such that $h$ satisfies
the conditions \eqref{hat.as.hom}--\eqref{finite.bound.U}.
The intuition behind our construction of the function $h$ is simple.
Since we consider asymptotically homogeneous Markov chain,
the chain behaves similar to the random walk with jumps like $\xi$.
We assume that limiting jump satisfies Cram\'er's condition,
hence it should be so that $h(i)\sim e^{\beta i}$ as $i\to\infty$.

Consider the transition kernel
$$
Q^{(\beta)}_N(i,j):=\frac{e^{\beta j}}{e^{\beta i}}P(i,j)\I\{j>N\},
$$
which is the result of exponential change of measure.
By the theorem conditions, $Q^{(\beta)}_N(i,\Zp)$ is finite for every $i$.
Let us find a level $N$ such that the kernel $Q^{(\beta)}_N$
satisfies the conditions of Theorem \ref{th:f.to.1}.

Denote
\begin{eqnarray*}
\delta_N(i) &:=& \log Q^{(\beta)}_N(i,\Zp)\\
&=& \log\sum_{j=N+1}^\infty e^{\beta(j-i)}P(i,j)\\
&=& \log \E\{e^{\beta\xi(i)};i+\xi(i)>N\}.
\end{eqnarray*}
Since
\begin{eqnarray*}
\sum_{i=0}^\infty |\E\{e^{\beta\xi(i)};i+\xi(i)>N\}-1|
&\le& \sum_{i=0}^\infty |\E e^{\beta\xi(i)}-1|
+\sum_{i=0}^\infty \E\{e^{\beta\xi(i)};i+\xi(i)\le N\}\\
&\le& \sum_{i=0}^\infty |\E e^{\beta\xi(i)}-1|
+\sum_{i=0}^\infty e^{-\beta(i-N)},
\end{eqnarray*}
we conclude by the condition \eqref{conver} that the condition
\eqref{delta} of Theorem \ref{th:f.to.1} holds for every $N\in\Zp$.

Further, $\delta_N^+(i)\to 0$ as $N\to\infty$, for every $i\in\Zp$.
Moreover, $\delta_N^+(i)\le\log^+\E e^{\beta\xi(i)}$ where the sum
$$
\sum_{i=0}^\infty \log^+\E e^{\beta\xi(i)}
$$
is finite, due to the condition \eqref{conver}.
Then the dominated convergence theorem yields the convergence
\begin{eqnarray*}
\delta_N:=\sum_{i=0}^\infty \delta_N^+(i) &\to& 0\quad\mbox{as }N\to\infty.
\end{eqnarray*}
So, if we prove that, for some $\gamma>0$,
\begin{eqnarray}\label{local.times.exp.gamma}
\sup_{i\in\Zp}\E_i e^{\gamma\ell^{(\beta)}_N(i)} &<& \infty,
\end{eqnarray}
then we may choose sufficiently large $N$ such that the condition
\eqref{local.times.exp} of Theorem \ref{th:f.to.1} holds;
here $\ell^{(\beta)}_N(i)$ is the local time at state $i$ of the
underlying chain $X^{(\beta)}_{N,n}$, $n\ge 0$, of the kernel
$Q^{(\beta)}_N$.

By Proposition \ref{pro:suffi.for.ell.m},
the relation \eqref{local.times.exp.gamma} will follow if we
find, for sufficiently large $N$, a minorant with positive mean
for the jumps $\xi^{(\beta)}_N(i)$ of the chain $X^{(\beta)}_{N,n}$.
The asymptotic homogeneity of the Markov chain $X_n$ implies that
\begin{eqnarray}\label{xi.N}
\xi^{(\beta)}_N(i) &\Rightarrow& \xi^{(\beta)}\quad\mbox{as }i\to\infty,
\end{eqnarray}
where the limiting random variable has distribution
$$
\P\{\xi^{(\beta)}=j\}=e^{\beta j}\P\{\xi=j\}
$$
with positive mean $\E\xi e^{\beta\xi}$. Therefore, there exist
a sufficiently large $N$ and a random variable $\eta$ with
positive mean, $\E\eta>0$, such that
$$
\xi^{(\beta)}_N(i) \ge_{st} \eta\quad\mbox{for all }i\in\Zp,
$$
and a minorant is identified.

Finally, the condition \eqref{lln} also follows from
minorization and the convergence, for every fixed $N$,
\begin{eqnarray*}
\P\{i+\eta_1+\ldots+\eta_n\ge N\mbox{ for all }n\ge 1\} &\to& 1
\quad\mbox{as }i\to\infty.
\end{eqnarray*}
So, for sufficiently large $N$, the kernel $Q^{(\beta)}_N$ satisfies
all the conditions of Theorem \ref{th:f.to.1}. Therefore,
there exists a harmonic function $f$ for this kernel
such that $f(i)\to 1$ as $i\to\infty$.

Let us consider the function $h(i):=e^{\beta i}f(i)$.
For every $i\in\Zp$, we have the equality
\begin{eqnarray*}
\sum_{j=N+1}^\infty P(i,j)h(j)
&=& \sum_{j=N+1}^\infty P(i,j)e^{\beta j}f(j)\\
&=& e^{\beta i}\sum_{j=N+1}^\infty Q^{(\beta)}_N(i,j)f(j)\\
&=& e^{\beta i}f(i) =h(i),
\end{eqnarray*}
so that $h(i)$ is the harmonic function for the Markov chain $X_n$
killed in $[0,N]$. Let us check that $h$ produces $\widehat X_n$
satisfying the conditions
\eqref{hat.as.hom}--\eqref{finite.bound.U}. First, the condition
\eqref{hat.as.hom} holds because, for every $j\in\Z$,
$$
\P\{\widehat\xi(i)=j\}
=\frac{e^{\beta(i+j)}f(i+j)}{e^{\beta i}f(i)}\P\{\xi(i)=j\}
\to e^{\beta j}\P\{\xi=j\}\quad\mbox{as }i\to\infty.
$$
Notice that the limiting random variable $\widehat\xi$
has mean $\E\xi e^{\beta\xi}$.

Second, let us prove that the condition \eqref{majoriz} holds.
From the upper bound
\begin{eqnarray*}
\P\{\widehat\xi(i)>j\} &=&
\sum_{k=j+1}^\infty \frac{f(i+k)}{f(i)} e^{\beta k}\P\{\xi(i)=k\}\\
&\le& c_1\E\{e^{\beta\xi(i)};\xi(i)>j\}\\
&\le& c_1\E\{e^{\beta\Xi};\Xi>j\},
\end{eqnarray*}
owing to the condition \eqref{Xi}, we deduce that
$\widehat\xi(i) \le_{st} \Xi_1$ where
\begin{eqnarray}\label{Xi1}
\E \Xi_1 &\le& c_2\sum_{j=0}^\infty \E\{e^{\beta\Xi};\Xi>j\}
\le c_2 \E\Xi e^{\beta\Xi}<\infty.
\end{eqnarray}
On the other hand,
\begin{eqnarray*}
\P\{\widehat\xi(i)<-j\} &=&
\sum_{k=j+1}^i \frac{f(i-k)}{f(i)} e^{-\beta k}\P\{\xi(i)=-k\}\\
&\le& c_3\sum_{k=j+1}^\infty e^{-\beta k}
\le c_4e^{-\beta j},
\end{eqnarray*}
so that $\widehat\xi(i)\ge -\Xi_2$ where $\Xi_2$ has some
positive exponential moment finite. Together with \eqref{Xi1}
this implies fulfillment of the condition \eqref{majoriz}
for the function $h(i)=e^{\beta i}f(i)$.

Third, the condition \eqref{finite.bound.U} follows from the equalities
\begin{eqnarray*}
\widehat U(i) &=& \sum_{n=0}^\infty\P\{\widehat X_n=i\}\\
&=& f(i)\sum_{n=0}^\infty \P\{X^{(\beta)}_{N,n}=i\}
=f(i)\E\ell^{(\beta)}_N(i)
\end{eqnarray*}
and from the bound \eqref{local.times.exp.gamma} for exponential
moments of the local times of the chain $X^{(\beta)}_{N,n}$.

Therefore, we may apply \eqref{pi.asy.V} and deduce that,
as $i\to\infty$,
\begin{eqnarray*}
\pi(i) &\sim& \frac{\sum_{j=0}^N\pi(j)h(j)}{\E\widehat\xi}
\frac{e^{-\beta i}}{f(i)}
\sim \frac{\sum_{j=0}^N\pi(j)e^{\beta j}f(j)}{\E\xi e^{\beta\xi}}
e^{-\beta i}.
\end{eqnarray*}
The proof of Theorem \ref{th:conv} is complete.
\end{proof}

We now turn to the case where $\E e^{\beta\xi(i)}$ converges to $1$
in a non-summable way. Our next result describes the behaviour of
$\pi$ in terms of a non-uniform exponential change of measure.
\begin{theorem}\label{th:non.gen}
Suppose that, for some $\varepsilon>0$,
\begin{eqnarray}\label{Xi.eps}
\sup_{i\in\Zp}\E e^{(\beta+\varepsilon)\xi(i)}<\infty.
\end{eqnarray}
Assume also that there exists a differentiable function $\beta(x)$
such that
\begin{eqnarray}\label{E.beta.1}
\sum_{i=0}^\infty |\E e^{\beta(i)\xi(i)}-1| &<& \infty,
\end{eqnarray}
and $|\beta'(x)|\le\gamma(x)$ where $\gamma(x)$ is a decreasing
integrable function of order $o(1/x)$. Then, for some $c>0$,
\begin{eqnarray*}
\pi(i) &\sim& ce^{-\int_0^i \beta(y)dy}\quad\mbox{as }i\to\infty.
\end{eqnarray*}
\end{theorem}
It should be noted that Theorem \ref{th:conv} can be seen as a special
case of Theorem \ref{th:non.gen} with $\beta(x)\equiv\beta$. But we
decided to split these statements, since the proof of
Theorem \ref{th:non.gen} contains a reduction to the case of summable
rate of convergence, which has been considered in Theorem \ref{th:conv}.

\begin{proof}[Proof of Theorem \ref{th:non.gen}]
Consider the function $g(x):=e^{\int_0^x \beta(y)dy}$ and
perform the corresponding change of measure:
$$
Q^{(g)}(i,j):=\frac{g(j)}{g(i)}P(i,j).
$$
First let us estimate
\begin{eqnarray*}
Q^{(g)}(i,\Zp)-1 &=& \frac{\E g(i+\xi(i))-g(i)}{g(i)}\\
&=& \E e^{\int_i^{i+\xi(i)}\beta(y)dy}-1.
\end{eqnarray*}

Observe that, with necessity, $\beta(i)\to\beta$ so that,
by the condition \eqref{Xi.eps},
\begin{eqnarray*}
\E\Bigl\{e^{\int_i^{i+\xi(i)}\beta(y)dy}-1;|\xi(i)|>\sqrt i\Bigr\}
&=& o(e^{-\varepsilon i/2})\quad\mbox{as }i\to\infty.
\end{eqnarray*}
Further, condition on the derivative of $\beta(y)$ implies that
\begin{eqnarray*}
\biggl|\int_i^{i+\xi(i)}\beta(y)dy-\beta(i)\xi(i)\biggr|
&\le& \int_i^{i+\xi(i)}|\beta(y)-\beta(i)|dy\\
&\le& \sup_{|y|\le\sqrt i}|\beta'(i+y)|\xi^2(i)/2\\
&\le& \gamma(i-\sqrt i)\xi^2(i)/2.
\end{eqnarray*}
Uniformly in $|\xi(i)|\le\sqrt i$, we have
$\gamma(i-\sqrt i)\xi^2(i)\le\gamma(i-\sqrt i)i\to 0$ as $i\to\infty$. Therefore,
again in view of the condition \eqref{Xi.eps},
\begin{align*}
\E\Bigl\{e^{\int_i^{i+\xi(i)}\beta(y)dy};|\xi(i)|\le\sqrt i\Bigr\}
&=\E e^{\beta(i)\xi(i)}+O(\gamma(i-\sqrt i)+e^{-\varepsilon i/2})\quad\mbox{as }i\to\infty.
\end{align*}
Hence,
\begin{eqnarray*}
|Q^{(g)}(i,\Zp)-1| &\le& |\E e^{\beta(i)\xi(i)}-1|
+O(\gamma(i-\sqrt i)+e^{-\varepsilon i/2}).
\end{eqnarray*}
Taking into account \eqref{E.beta.1} and that the sequence
$\gamma(i-\sqrt i)$ is summable, we obtain that
\begin{eqnarray*}
\sum_{i=0}^\infty|Q^{(g)}(i,\Zp)-1| &<& \infty.
\end{eqnarray*}
This allows to apply Theorem \ref{th:f.to.1} to the kernel
$Q^{(g)}$ killed in some set $[0,N]$ in the same way
as in the proof of Theorem \ref{th:conv} and to deduce that
$\pi(i)\sim c/g(i)$ as $i\to\infty$, which completes the proof.
\end{proof}

Since $\beta(x)$ is not given explicitly, Theorem \ref{th:non.gen}
can not be seen as a final statement. For this reason we describe below
some situation where $\beta$ can be expressed via the difference
$\E e^{\beta \xi(i)}-1$.
\begin{corollary}\label{cor:non.sum}
Assume the condition \eqref{Xi.eps} and that there exists
a differentiable function $\alpha(x)$ such that
$\alpha'(x)$ is regularly varying at infinity with index $-2<r<-3/2$ and
\begin{eqnarray}\label{non.conver}
\E e^{\beta\xi(i)}-1 &=& \gamma(i)+\alpha(i),
\end{eqnarray}
where $\sum_{j=0}^\infty|\gamma(i)|<\infty$. Suppose also that
\begin{eqnarray}\label{conver.of.E}
\sum_{j=0}^\infty|\alpha(i)(\E\xi(i)e^{\beta\xi(i)}-m)| &<& \infty,
\end{eqnarray}
where $m:=\E\xi e^{\beta\xi}$. Then
\begin{eqnarray}\label{answer.1}
\pi(i) &\sim& c e^{-\beta i+A(i)/m}\quad\mbox{as }i\to\infty,
\end{eqnarray}
where $c>0$ and $A(x):=\int_0^x\alpha(y)dy$.
\end{corollary}

\begin{proof}
Notice that, since $-2<r<-3/2$, $A(x)\to\infty$, $A(x)=o(x)$
as $x\to\infty$ and $\sum_{i=1}^\infty\alpha^2(i)<\infty$.

Take $\beta(x):=\beta-\alpha(x)/m$. Since $r<-3/2$,
$\alpha(i)=o(1/\sqrt i)$. Hence, by Taylor's theorem,
uniformly in $|\xi(i)|\le\sqrt i$,
\begin{eqnarray*}
e^{-\alpha(i)\xi(i)/m} &=& 1-\alpha(i)\xi(i)/m+O(\alpha^2(i)\xi^2(i)).
\end{eqnarray*}
which yields
\begin{eqnarray*}
\E e^{\beta(i)\xi(i)} &=& \E e^{\beta\xi(i)}
-\alpha(i)\E\xi(i)e^{\beta\xi(i)}/m+O(\alpha^2(i))\\
&=& \E e^{\beta\xi(i)}-\alpha(i)
+O(|\alpha(i)(\E\xi(i)e^{\beta\xi(i)}-m)|+\alpha^2(i))\\
&=& 1+\gamma(i)+O(|\alpha(i)(\E\xi(i)e^{\beta\xi(i)}-m)|+\alpha^2(i)).
\end{eqnarray*}
Thus, the function $\beta(x)$ satisfies all the conditions
of Theorem \ref{th:non.gen} and the proof is complete.
\end{proof}

Notice that the key condition on the rate of convergence of
$\E e^{\beta\xi(i)}$ to $1$ that implies asymptotics
\eqref{answer.1} in the latter corollary is that the sequence
$\alpha^2(i)$ is summable. If it is not so, that is, if the
index $r+1$ of regular variation of the function $\alpha(x)$
is between $-1/2$ and $0$, then the asymptotic behaviour of
$\pi(i)$ is different from \eqref{answer.1}
which is specified in the following corollary.

\begin{corollary}\label{cor_power}
Assume the condition \eqref{Xi.eps} and that there exists
a differentiable function $\alpha(x)$ such that
$$
|\alpha(x)|\le \frac{c}{(1+x)^{\frac{1}{M+1}+\varepsilon}}
$$
for some $c<\infty$, $M\in\N$, and $\varepsilon>0$,
\begin{equation}\label{cor:integr}
|\alpha'(x)|\le \gamma_1(x)
\end{equation}
for some decreasing integrable $\gamma_1(x)$
and
$$
\E e^{\beta\xi(i)}-1=\alpha(i)+\gamma_2(i),\quad i\ge 0,
$$
where $\sum_{i=0}^\infty\gamma_2(i)<\infty$.
Assume also that, for every $k=1$, $2$, \ldots, $M$,
\begin{equation}\label{cor.1}
m_k(i)=m_k+\sum_{j=1}^{M-k} D_{k,j}\alpha^j(i)+O(\alpha^{M-k+1}(i)),
\end{equation}
where $m_k(i):=\E\xi^k(i)e^{\beta\xi(i)}$ and
$m_k:=\E\xi^ke^{\beta\xi}$. Then there exist real numbers
$R_1$, $R_2$, \ldots, $R_M$ such that
\begin{equation}
\label{cor.2}
\pi(i)\sim c\exp\biggl\{-\beta i-\sum_{k=1}^M R_k\int_0^i\alpha^k(x)dx\biggr\}
\quad\mbox{as }i\to\infty.
\end{equation}
\end{corollary}

\begin{proof}
Define
$$
\Delta(x):=\sum_{k=1}^M R_k\alpha^k(x).
$$
In view of Theorem \ref{th:non.gen} it suffices to show that there exist $R_1,R_2,\ldots,R_M$ such that
\begin{equation}
\sum_{i=1}^\infty\left|\E e^{(\beta+\Delta(i))\xi(i)}-1\right|<\infty.
\end{equation}
Indeed, $\Delta(x)$ is differentiable and $|\Delta'(x)|\le C|\alpha'(x)|$.
Therefore, we may apply Theorem \ref{th:non.gen}
with $\beta(x)=\beta+\Delta(x)$.

By Taylor's theorem, the calculations similar
to the previous corollary show that, as $i\to\infty$,
\begin{align*}
\E e^{(\beta+\Delta(i))\xi(i)}
&=\E e^{\beta\xi(i)}+\sum_{k=1}^M\frac{m_k(i)}{k!}\Delta^k(i)+
O(\Delta^{M+1}(i))\\
&=1+\alpha(i)+\gamma_2(i)+\sum_{k=1}^M\frac{m_k(i)}{k!}\Delta^k(i)
+O(\alpha^{M+1}(i)+e^{-\varepsilon i/2}).
\end{align*}
From this equality we infer that we may determine $R_1$, $R_2$,
\ldots, $R_M$ by the relation
\begin{align}\label{cor.3}
\alpha(i)+\sum_{k=1}^M\frac{m_k(i)}{k!}\Delta^k(i)
=O(\alpha^{M+1}(i)).
\end{align}
It follows from the assumption \eqref{cor.1} and the bound
$\Delta(x)=O(\alpha(x))$ that \eqref{cor.3} is equivalent to
$$
z+\sum_{k=1}^M\frac{1}{k!}\biggl(m_k+\sum_{j=1}^{M-k}D_{k,j}z^j\biggr)
\biggl(\sum_{j=1}^M R_jz^j\biggr)^k=O(z^{M+1})
\quad\mbox{ as }z\to 0.
$$
Consequently, the coefficient at $z^k$ should be zero for every
$k\le  M$, and we can determine all $R_k$ recursively.
For example, the coefficient at $z$ equals $1+m_1R_1$.
Thus, $R_1=-1/m_1$. Further, the coefficient
at $z^2$ is $D_{1,1}R_1+m_1R_2+m_2R_1^2/2$ and, consequently,
$$
R_2=\frac{D_{1,1}}{m_1^2}-\frac{m_2}{2m_1^3}.
$$
All further coefficients can be found in the same way.
\end{proof}

If $\alpha(x)$ from Corollary \ref{cor_power}  decreases slower than
any power of $x$ but \eqref{cor:integr} and \eqref{cor.1} remain
valid, then one has, by the same arguments,
$$
\pi(i)=\exp\biggl\{-\beta i-\sum_{k=1}^M R_k\int_0^i\alpha^k(x)dx
+O\left(\int_0^i\alpha^{M+1}(x)dx\right)\biggr\}
$$
which can be seen as a corrected logarithmic asymptotic for $\pi$.
To obtain precise asymptotics one needs more information on the
moments $m_k(i)$.
\begin{corollary}\label{cor_full}
Assume the condition \eqref{Xi.eps} and that there exists
a differentiable function $\alpha(x)$ such that
\eqref{cor:integr} holds,
\begin{equation}\label{cor.0.full}
\E e^{\beta\xi(i)}-1=\alpha(i),\quad i\ge 1
\end{equation}
and
\begin{equation}\label{cor.1.full}
m_k(i)=m_k+\sum_{j=1}^\infty D_{k,j}\alpha^j(i)
\end{equation}
for all $k\ge 1$. Assume furthermore that
$\sup_{k\ge 1}\sum_{j=1}^\infty D_{k,j}r^j<\infty$ for some $r>0$.
Then there exist real numbers
$R_1$, $R_2$, \ldots,
$$
\pi(i)\sim c\,\exp\biggl\{-\beta i-\sum_{k=1}^\infty R_k\int_0^i\alpha^k(x)dx\biggr\}.
$$
\end{corollary}

\begin{proof}
For every $i\ge 1$ let $\beta(i)$ denote the positive solution of the equation
$$
\E e^{\beta(i)\xi(i)}=1.
$$
Since $\E e^{\beta\xi(i)}$ is finite for all $\gamma\le \beta+\varepsilon$, we may
rewrite the latter equation as Taylor's series:
$$
\E e^{\beta\xi(i)}+\sum_{k=1}^\infty\frac{\Delta^k(i)}{k!}\E\xi^k(i)e^{\beta\xi(i)}=1,
$$
where $\Delta(i)=\beta(i)-\beta$. Taking into account \eqref{cor.0.full} and
\eqref{cor.1.full}, we then get
\begin{equation}\label{cor.2.full}
\alpha(x)+\sum_{k=1}^\infty\frac{\Delta^k(x)}{k!}\sum_{j=0}^\infty D_{k,j}\alpha^j(x)=0,\quad x\ge 0.
\end{equation}
Set $D_{0,1}=1$ and define
$$
F(z,w):=\sum_{k,j\ge 0}\frac{D_{k,j}}{k!}z^jw^k.
$$
Therefore, \eqref{cor.2.full} can be written as $F(\alpha(x),\Delta(x))=0$. In other words,
we are looking for a function $w(z)$ satifying $F(z,w(z))=0$. Since $F(0,0)=0$ and
$\frac{\partial}{\partial w}F(0,0)=m_1>0$, we may apply Theorem B.4 from Flajolet and Sedgewick
\cite{FS} which says that $w(z)$ is analytic in a vicinity of zero, that is, there exists $\rho>0$
such that
$$
w(z)=\sum_{n=1}^\infty R_nz^n,\quad |z|<\rho.
$$
Consequently,
$$
\Delta(x)=\sum_{n=1}^\infty R_n\alpha^n(x)
$$
for all $i$ such that $|\alpha(i)|<\rho$.

Applying Theorem \ref{th:non.gen} with $\beta(x)=\beta+\Delta(x)$, we get
$$
\pi(i)\sim ce^{-\beta(i)-\int_0^i\Delta(y)dy}.
$$
Integrating $\Delta(y)$ piecewise, we complete the proof.
\end{proof}

We finish with the following remark. In the proof of Corollary \ref{cor_full}
we have adapted the derivation of the Cram\'er series in large deviations for sums
of independent random variables, see, for example, Petrov \cite{Pet75}. There
is just one difference: we needed analyticity of an implicit function instead
of analyticity of an inverse function.

\affiliationone{Denis Denisov\\
   School of Mathematics\\
   University of Manchester\\
   Oxford Road, Manchester M13 9PL, UK
   \email{denis.denisov@manchester.ac.uk}}
\affiliationtwo{Dmitry Korshunov\\
   Sobolev Institute of Mathematics\\
   Koptyuga pr. 4, Novosibirsk 630090, Russia
   \email{korshunov@math.nsc.ru}}
\affiliationthree{Vitali Wachtel\\
   Mathematical Institute\\
   University of Munich\\
   Theresienstr. 39, Munich 80333, Germany
   \email{wachtel@math.lmu.de}}

\end{document}